\documentclass[12pt,reqno]{amsart}
\usepackage{amsmath,amssymb,amsthm}
\usepackage[mathscr]{eucal}

\begin{document}
\newtheorem{thm}{Theorem}
\numberwithin{thm}{section}
\newtheorem{lemma}[thm]{Lemma}
\newtheorem{remark}[thm]{Remark}
\newtheorem{corr}[thm]{Corollary}
\newtheorem{proposition}[thm]{Proposition}
\newtheorem{theorem}[thm]{Theorem}
\newtheorem{deff}[thm]{Definition}
\newtheorem{case}[thm]{Case}
\newtheorem{prop}[thm]{Proposition}
\numberwithin{equation}{section}
\newcommand{\uG}{\underline{G}}
\newcommand{\bD}{\mathrm{I\! D\!}}
\newcommand{\bR}{\mathbb R}
\newcommand{\uT}{\underline{T}}
\newcommand{\uB}{\underline{B}}
\newcommand{\uU}{\underline{U}}
\newcommand{\bH}{\mathrm{I\! H\!}}
\newcommand{\uA}{\underline{A}}
\newcommand{\uM}{\underline{M}}
\newcommand{\uN}{\underline{N}}
\newcommand{\uP}{\underline{P}}
\newcommand{\sN}{\cal N}
\newcommand{\bC}{\mathbb C}
\newtheorem{corollary}[thm]{Corollary}
\newtheorem{others}{Theorem}
\newcommand{\nucleo}{p_{D}_{t}(x,y)}
\newtheorem{conjecture}[thm]{Conjecture}
\newtheorem{definition}[thm]{Definition}
\newtheorem{cl}{Claim}
\newtheorem{cor}{Corollary}
\newcommand{\ds}{\displaystyle}
\date{}
\newcommand{\pa}{{\cal P}_{\alpha}}
\newcommand{\re}{{\rm Re}}
\newcommand{\im}{{\rm Im}}

\title[Conformal martingales]{Subordination by conformal  
martingales in $L^{p}$ and zeros of Laguerre polynomials}
\author{Alexander Borichev}\address{Alexander Borichev, 
Universit\'e de Provence, Marseille,\newline\noindent
{\tt borichev@cmi.univ-mrs.fr}}
\author{Prabhu Janakiraman}\address{Prabhu Janakiraman, 
Department of Mathematics, Michigan State University,
{\tt pjanakir1978@gmail.com}}
\author{Alexander Volberg}\address{Alexander Volberg, 
Department of Mathematics, Michigan State University,
{\tt volberg@math.msu.edu}}

\thanks{The research of the first author was partially supported 
by the ANR grants DYNOP and FRAB; the research of the second and 
the third authors was partially supported by the NSF grants 
DMS-0758552 and DMS-0605166.}

\begin{abstract} 
Given martingales $W$ and $Z$ such that $W$ is differentially subordinate to $Z$, 
Burkholder obtained the sharp martingale inequality $E|W|^p \le  (p^*-1)^pE|Z|^p$, where $p^* = \max\{p, \frac{p}{p-1}\}$. What happens if one of the martingales is also a conformal martingale? Ba\~nuelos and Janakiraman proved that if $p\geq 2$ and $W$ is a conformal martingale differentially subordinate to any martingale $Z$, then $E|W|^p \leq [(p^2-p)/2]^{p/2}E|Z|^p$. In this paper, we establish that if $p\geq 2$, $Z$ is conformal, and $W$ is any martingale subordinate to $Z$, then 
$\mathbb E|W|^p\le [\sqrt2(1-z_p)/z_p]^p\mathbb E|Z|^p$, where $z_p$ is the smallest positive zero of a certain solution of the Laguerre ODE. We also prove the sharpness 
of this estimate, and an analogous one in the 
dual case for $1<p<2$.
Finally, we give an application of our results.
Previous estimates on the $L^p$ norm of the Beurling--Ahlfors transform give at best $\|B\|_p\lesssim \sqrt{2}\,p$ as $p\rightarrow\infty$. We improve this to $\|B\|_p\lesssim 1.3922 \,p$ as $p\rightarrow\infty$.
\end{abstract}

\maketitle

\section{Introduction}

In this paper we address the question of finding the best 
$L^p$-norm constant for martingale transforms with one-sided 
conformality. Let $\mathcal O=(\Omega,\mathcal B,P)$ be a 
probability space with filtration $\mathcal{B}$ generated by a 
two-dimensional Brownian motion $B(t)$. Let 
$X(t)=\int_0^t \nabla X(s)\cdot dB(s)$ and 
$Y(t)=\int_0^t \nabla Y(s)\cdot 
dB(s)$ 
be two $\mathbb R^2$-valued martingales on this probability space, 
such that the quadratic variation of $Y$ runs slower than the 
quadratic variation of $X$, i.e. 
$d\langle Y\rangle(s)\le d\langle X\rangle(s)$, or equivalently
\begin{multline*}
|\nabla Y(s)|=\sqrt{|\nabla Y_1(s)|^2+|\nabla Y_2(s)|^2}\\ \le 
\sqrt{|\nabla X_1(s)|^2+|\nabla X_2(s)|^2}=|\nabla X(s)|,\qquad s\ge 0,
\end{multline*}
where $X=(X_1,X_2)$, $Y=(Y_1,Y_2)$.
By definition, $Y$ is said to be differentially subordinate to 
$X$ or to be a martingale transform of $X$. If, for $1<p<\infty$, 
we have $\mathbb E|X(t)|^p<\infty$, then the Burkholder--Davis--Gundy and 
Doob inequalities (see \cite{RoWi}) imply that $\mathbb E|Y(t)|^p<\infty$ 
and there exists a universal constant $C_p$ such that 
$\|Y(t)\|_p\le C_p\|X(t)\|_p$. We use the notation 
$\|X\|_p=\|X(t)\|_p=(\mathbb E|X(t)|^p)^{1/p}$. 
An evident problem then is to find the best 
constant $C_p$.

Burkholder solved this problem completely in a series of 
papers in the 1980's, see in particular \cite{Bu1} and 
\cite{Bu3}. He proved that 
$$
C_p = p^*-1, \qquad p^*=\max\{p,p'\},\quad p'=\frac{p}{p-1}.
$$
His approach (used also in the present paper) is as follows (see 
\cite[Section 5]{Bu1} for a more general viewpoint). Consider 
the function $V(x,y) = |y|^p-C_p^p|x|^p$, where $|{\cdot}|$ is the Euclidean norm in 
$\mathbb R^2$; we wish to find $C_p$ 
such that for martingales $X$ and $Y$ as above, we always have 
$\mathbb E\,V(X,Y)\le 0$. Now find (if it exists) a majorant 
function $U(x,y)\geq V(x,y)$ such that $U(0,0)=0$ and $U(X,Y)$ 
is a supermartingale; such a function must exist for the optimal 
$C_p$, see Section 2. Then we have
$$
\mathbb E\,V(X,Y)\le \mathbb E\,U(X,Y)\le 0.
$$
Burkholder shows that when $C_p=p^*-1$ such a majorant exists 
and equals
$$
U(x,y) = p\Bigl(1-\frac{1}{p^*}\Bigr)^{p-1}
\bigl(|y|-(p^*-1)|x|)(|x|+|y|\bigr)^{p-1},
$$
and he finds extremal functions (extremals) to show that $p^*-1$ is in fact the best 
(least) possible constant. Generally, to show that $U(X,Y)$ is a 
supermartingale we need to verify that $U$ is a supersolution 
for a family of PDEs; in this case, it suffices to show that $U$ 
is a biconcave function. Thus Burkholder translates 
martingale $L^p$ problems to the calculus-of-varia\-tions setting 
and solves the corresponding obstacle problems. In other work 
\cite{Bu8}, he also shows that this martingale problem and its 
solutions are related to the special nature of the range space of 
the martingales, and obtains specific geometric characterization 
of all Banach spaces that have finite martingale-transform 
constant.

\section{Burkholder, Bellman and Beurling--Ahlfors}
\label{se2}

One of the primary applications for Burkholder's theorem has 
come in Fourier analysis in estimating the $L^p$ norm of the 
Beurling--Ahlfors transform.

The Beurling--Ahlfors transform is a singular integral operator acting on $L^p(\bC)$ and defined by  
$$
B\varphi(z) = \frac{1}{\pi}\textrm{p.v.}\int_{\bC}\frac{f(w)}{(z-w)^2}dm(w).
$$
This self-adjoint operator arises naturally in the quasiconformal mapping theory  due to the way it relates the complex derivative operators. If $\partial = \frac{\partial_x - i \partial_y}{2}$ and $\bar{\partial} = \frac{\partial_x + i \partial_y}{2}$, then 
$$ 
B = \frac{\partial}{\bar{\partial}} = \frac{\partial^2}{\Delta}.
$$
An alternative representation in terms of the second order Riesz transforms \cite{Du} is particularly important for us:
$$
B = R_2^2 - R_1^2 -i 2 R_1R_2.
$$
One of the fundamental open problems for this operator is the computation of its $L^p$ norm $\|B\|_p$. This question gains prominence due to the information it would yield regarding the Beltrami equation (see \cite{NV1}) and for the proof of the former Gehring-Reich conjecture (and presently Astala's theorem) \cite{Ast}. Presently the quest for $\|B\|_p$ attracts mathematicians in different areas of analysis and probability. It remains unsolved.

It is a conjecture by Iwaniec \cite{Iw} that the norm 
constant is $\|B\|_p = p^*-1$, the same constant as in 
Burkholder's theorem for martingales; by duality it is known that $\|B\|_p=\|B\|_{p'}$. The lower bound (first found by Lehto \cite{Le}) can be proved by finding a suitable sequence of functions $\{\varphi_j\}$ such that $\lim_{j\rightarrow\infty}\frac{\|B\varphi_j\|_p}{\|\varphi_j\|_p} = p^*-1$. The upper bound is still an open problem. Estimates have been obtained and gradually improved upon, relying on some critical theorems of Burkholder in the martingale theory, see 
\cite{Bu1,Bu2,BaWa1,NV1,BaMH}. 

The first major 
breakthrough in finding the connection between martingale 
estimates and the Beurling--Ahlfors operator came in 
\cite{BaWa1} where Ba\~nuelos and Wang show that if a function 
$f\in L^p({\mathbb{R}}^2)$ is extended harmonically as $U_f(x,t)$ to 
the upper half-space ${\mathbb{R}}^3_+$, then for the martingale $X_t=U_f(B_t)$ there exists a martingale transform $Y_t$ satisfying (essentially)
$$ 
X_\tau\approx f(x),\quad\mathbb E\,[Y_\tau|B_\tau=x]=Bf(x),\quad d\langle Y\rangle\le 16\,d\langle X\rangle.
$$
Here $B_t$ is 3-dimensional Brownian motion, $\tau$ is its exit 
time from ${\mathbb{R}}^3_+$, and the conditional expectation 
$\mathbb E\,[Y_\tau|B_\tau=x]$ is the average value of $Y_\tau$ 
over paths that exit at $x$. This then implies (essentially)
$$
\|Bf\|_p =\|\mathbb E\,[Y_\tau|B_\tau=x]\|_p\le \|Y_\tau\|_p 
\le 4(p^*-1)\|X_\tau\|_p \le 4(p^*-1)\|f\|_p.
$$ 
The first inequality follows from Jensen's and the second one follows from 
Burk\-holder's theorem. Thus we have $\|B\|_p \le 4(p^*-1)$.

In a series of papers starting in the late 1990's (\cite{NT,NV1,NTV8,DV1,PV,P}) it is shown that the martingale/obstacle problem 
treated by Burkholder fits within a general framework derived 
from Stochastic Control theory, which also works with other 
questions in harmonic analysis. Here again, a special function 
$\mathcal{B}$ called the Bellman function is found in relation 
to the problem, and it usually satisfies certain concavity and 
boundedness conditions. Burkholder's function is an 
example of a Bellman function. In fact, the Bellman function 
theory establishes that such a function $\mathcal{B}$ 
necessarily exists for the corresponding optimization problem, 
and its concavity and boundedness 
properties are sharp on the extremals. Using the Bellman function approach, Nazarov and 
Volberg \cite{NV1} obtain a better estimate $\|B\|_p\le 2
(p^*-1)$. We describe how this was done. Given $f\in L^p$ and 
$g\in L^{p'}$, denote their heat extensions to the upper 
half-space by $f$ and $g$ again; we can show that 
\begin{gather*}
\Bigl|\int_{\bC} Bf\cdot g\Bigr|= 
\Bigl|2\int_{\bR^3_+} (\partial_x+i\partial_y)f(\partial_x+i\partial_y)g\,dxdydt\Bigr| \\
\le 2\int_{\bR^3_+} (|\partial_xf||\partial_xg|+|\partial_yf||\partial_yg|+|\partial_xf||\partial_yg|+|\partial_yf||\partial_xg|)\,dxdydt.
\end{gather*}
We wish to bound this integral from above by 
$c(p^*-1)\|f\|_p\|g\|_{p'}$. However, we do not know how to 
integrate terms like $|\partial_xf||\partial_xg|$, so the idea 
is to find another function above it which can be integrated and 
whose integral has the required upper bound. Now construct 
(for $p>2$) the Bellman function $\mathcal{B}$ defined on the 
domain 
$$
D_p=\{(\xi,\eta,X,Y)\subset (\bR_+)^2\times (\bR_+)^2\times 
\bR_+\times \bR_+: X>|\xi|^p,Y>|\eta|^q\},
$$ 
that satisfies (essentially)
\begin{enumerate}
\item $0\le \mathcal B \le (p-1)X^{1/p}Y^{1/q}$, \\
\item $-\langle d^2\mathcal{B} \,d\xi,d\eta\rangle\ge 2|d\xi|\cdot|d\eta|.$
\end{enumerate}
The actual construction of (or an existence proof for) 
$\mathcal{B}$ involves taking supremum of appropriate functions 
over certain families of martingales, similar to how Burkholder 
defines his function in \cite{Bu1}. For more details on 
Bellman function constructions see \cite{NTV8}, \cite{VaVo}, 
\cite{VaVo2}.

Define $b:\bR^2\times\bR_+\rightarrow \bR_+$ by 
$b(x,t)=\mathcal{B}(f,g,|f|^p,|g|^p)$ where all input functions 
are the heat extensions. Let $v=(f,g,|f|^p,|g|^p)$. The 
boundedness condition on $\mathcal{B}$ implies
\begin{multline*}
4\pi R^2b(0,R^2) \le (p-1)\left(\int |f|^pe^{\frac{-|x|^2}{4R^2}}\right)^{1/p}\left(\int |g|^qe^{\frac{-|x|^2}{4R^2}}\right)^{1/q} \\ \rightarrow (p-1)\|f\|_p\|g\|_q.
\end{multline*}
Some clever analysis shows that $4\pi R^2b(0,R^2)$ is asymptotically (as $R\rightarrow\infty$) bounded below by
\[\int \left(\langle -d^2\mathcal{B}\,\partial_xv,\partial_xv\rangle
+\langle -d^2\mathcal{B}\,\partial_yv,\partial_yv\rangle\right).\]
By the concavity condition on $\mathcal{B}$, the latter expression is bounded below by
$$
\int_{\bR^3_+} (|\partial_xf||\partial_xg|+|\partial_yf||\partial_yg|+|\partial_xf||\partial_yg|+|\partial_yf||\partial_xg|)\,dxdydt.
$$
Thus we conclude that for $p\ge 2$ we have 
$|\int Bf\cdot g|\le 2(p-1)\|f\|_p\|g\|_q$. The result 
for $1<p<2$ follows by duality.

Following \cite{NV1}, Ba\~nuelos and M\'endez \cite{BaMH} redo the work done in \cite{BaWa1} but this time with heat extensions and space-time Brownian motion and also obtain $\|B\|_p\le 2(p^*-1)$.

\section{Conformal martingales and the Beurling--Ahlfors transform}

A complex-valued martingale $Y=Y_1+iY_2$ is said to be 
{\it conformal} if the quadratic variations of the coordinate martingales are equal and their mutual covariation is $0$:
$$
d\langle Y_1\rangle =d\langle Y_2\rangle,\quad 
d\langle Y_1,Y_2\rangle=0.
$$
In \cite{BaJ1}, Ba\~nuelos and Janakiraman make the observation 
that the martingale associated with the Beurling--Ahlfors 
transform is in fact a conformal one. They show that 
Burkholder's proof in \cite{Bu3} naturally accommodates for this 
property and leads to an improvement in the estimate of 
$\|B\|_p$.

\begin{theorem}\label{BurkBaJa} {\rm (}One-sided conformality 
treated by Burkholder's me\-thod{\rm )}
\begin{enumerate}
\item {\rm (}Left hand side conformality{\rm )} Suppose that 
$2\le p<\infty$. If $Y$ is a conformal martingale and $X$ is 
any martingale such that $d\langle Y\rangle\le d\langle X\rangle$, then 
$$
\|Y\|_p \le \sqrt{\frac{p^2-p}{2}}\|X\|_p.
$$
\item {\rm (}Right hand side conformality{\rm )} Suppose that 
$1<p\le 2$. If $X$ is a conformal martingale and $Y$ is any 
martingale such that $d\langle Y\rangle\le d\langle X\rangle$, 
then 
$$
\|Y\|_p \le \sqrt{\frac{2}{p^2-p}}\|X\|_p.
$$
\end{enumerate}
\end{theorem}

It is not known whether these estimates are optimal.

The result for the right hand side conformality is actually stated in \cite{BJV}. It follows the same lines of 
proof as that for the left hand side conformality.  If $X$ and $Y$ are 
the martingales associated with $f$ and $Bf$ respectively, then 
$Y$ is conformal, $d\langle Y\rangle\le 4d\langle X\rangle$ and, 
hence, by (1) we obtain
\begin{equation}\label{Beurest}
\|Bf\|_p\le \sqrt{2(p^2-p)}\|f\|_p,\qquad p\ge 2.
\end{equation}
Interpolating between this estimate $\sqrt{2(p^2-p)}$ and the 
known one $\|B\|_2=1$, Ba\~nuelos and Janakiraman \cite{BaJ1} establish the 
present best published estimate:
\begin{equation}\label{best-est}
\|B\|_p\le 1.575 (p^*-1).
\end{equation}

At the end of the paper we prove a slightly better asymptotic estimate 
(Theorem~\ref{main_thm}):
$$
\limsup_{p\to\infty}\frac{\|B\|_p}{p} \le 1.3922\,,\,\,\text{and}\,\, \|B\|_p \le 1.4\,p\,,\,\,\text{if}\,\, p\ge 1000\,.
$$

\section{New Questions and Main Results}

Since $B$ is associated with the left hand side conformality and since we know that $\|B\|_p = \|B\|_{p'}$, two important questions are
\begin{enumerate}
\item If $2< p<\infty$, what is the best constant $C_p$ in the left hand side conformality problem: $\|Y\|_p\le C_p\|X\|_p$, where $Y$ is conformal and 
$d\langle Y\rangle \le d\langle X\rangle $?
\item Similarly, if $1<p'<2$, what is the best constant $C_{p'}$ in the left hand side conformality problem?
\end{enumerate}
We have separated these two questions since Burkholder's proof (and his function) 
already gives an improvement in the conformal case when $p\geq 2$. However, no 
estimate (better than $p-1$) follows from analyzing Burkholder's function when 
$1<p'<2$ in the conformal case. One could hope that $C_{p'}<\sqrt{\frac{p^2-p}{2}}$ 
when $1<p'=\frac{p}{p-1}<2$. This paper ``answers" this hope in the negative by 
finding $C_{p'}$; see Theorem~\ref{MAINThm}. We also pose and answer the 
analogous question of right hand side conformality when $2<p<\infty$. In the spirit of 
Burkholder \cite{Bu8}, we believe that these questions are of independent interest in 
the martingale theory and may have deeper connections with other areas of 
mathematics.

Given $p>1$, denote by $z_{p}$ is the least positive root in $(0,1)$ of the bounded Laguerre function $L_{p}$.

\begin{theorem}\label{MAINThm}Let $Y=(Y_1,Y_2)$ be a conformal martingale and $X=(X_1,X_2)$ be an arbitrary martingale.
\begin{enumerate}
\item Let $1<p'\le 2$. Suppose $d\langle Y\rangle\le d\langle X\rangle$. Then the best constant in the inequality $\|Y\|_{p'}\le C_{p'}\|X\|_{p'}$ is 
\begin{equation}\label{constp<2}
C_{p'} = \frac{1}{\sqrt{2}}\frac{z_{p'}}{1-z_{p'}}.
\end{equation}
\item Let $2\le p<\infty$. Suppose $d\langle X\rangle\le 
d\langle Y\rangle$. Then the best constant in the inequality $\|X\|_{p}\le C_{p}\|Y\|_{p}$ is 
\begin{equation}\label{constp>2}
C_{p} = \sqrt{2}\frac{1-z_{p}}{z_{p}}.
\end{equation}
\end{enumerate}
\end{theorem}

The Laguerre function $L_p$ solves the ODE 
\[sL_p''(s)+(1-s)L_p'(s)+pL_p(s) = 0.\] These functions are discussed further on and their properties are reviewed in Section \ref{laguerresection}; see also \cite{BJV}, \cite{C}, \cite{CL}.

For asymptotics of $z_p$, $C_p$, $C_{p'}$ as $p\to\infty$ see 
Section~\ref{Asymptotics}. In particular, 
$\lim_{p\to\infty}C_{p'}/C_p>1$.

Before we embark on the proof of Theorem~\ref{MAINThm}, let us mention that there is also the question of two sided conformality: what is the best constant when both $X$ and $Y$ are conformal martingales? This problem is solved by the authors for $2<p<\infty$ in \cite{BJV} (and, recently, by 
Ba\~nuelos and Osekowski for $0<p<2$ in \cite{BO}), and the answer is $C_p = \frac{1+z_p}{1-z_p}$ where $z_p$ is the largest root in $[-1,1]$ of the Legendre function $F$ solving $(1-s^2)F''-2sF'+pF=0$. For large $p$ we have then $C_p<\sqrt{\frac{p^2-p}{2}}$.

\section{Proof of Theorem~\ref{MAINThm}: Right hand side conformality, $2<p<\infty$}

Let $X=(X_1,X_2)$ be an arbitrary martingale, and let $Y=(Y_1,Y_2)$ be a conformal martingale: 
$d\langle Y_1\rangle=d\langle Y_2\rangle$ and 
$d\langle Y_1,Y_2\rangle=0$. Let us assume that
$$
d\langle X\rangle \le d\langle Y_1\rangle= \frac{1}{2}
d\langle Y\rangle;
$$
then the constant corresponding to \eqref{constp>2} is
$$
\frac{1-z_p}{z_p}.
$$
Let $\tilde{V}(x,y)=|x|^p-c^p|y|^p$. Our objective is to find the best constant $c$ for which there exists a minimal majorant $\tilde{U}(x,y)\ge \tilde{V}(x,y)$, $\tilde{U}(0,0) \le 0$, such  that for $X$ and $Y$ as above, the process $\tilde{U}(X,Y)$ is a supermartingale. It follows then that $\mathbb E[\tilde{V}(X,Y)]\le  \mathbb E[\tilde{U}(X,Y)]\le 0$.
Furthermore, this supermartingality condition on $\tilde{U}(X,Y)$ is equivalent (by appealing to It$\hat{\textrm{o}}$'s formula) to the property that the quadratic form generated by $\tilde{U}$ 
is negative (in the distribution sense), i.e. 
\begin{equation}\label{cond-U1}
\sum_{i,j=1}^2 \tilde{U}_{x_ix_j}d\langle X_i,X_j\rangle+\Delta_y\tilde{U}d\langle Y_1\rangle+\sum_{i,j=1}^2 2\tilde{U}_{x_iy_j}d\langle X_i,Y_j\rangle\le 0.
\end{equation}
As in the Legendre case \cite{BJV}, the functions $\tilde{U}$, 
$\tilde{V}$ only depend on $|x|$ and $|y|$, hence 
\begin{gather*}
\tilde{U}(x_1,x_2,y_1,y_2)=U\Bigl(\sqrt{x_1^2+x_2^2},\sqrt{y_1^2+y_2^2}\Bigr)=
U(|x|,|y|),\\
\tilde{V}(x_1,x_2,y_1,y_2)=|x|^p-c^p|y|^p=V(|x|,|y|).
\end{gather*}
Let us introduce the vectors:
\begin{gather*}
h_1 :=\frac{(x_1,x_2)\cdot (\nabla X_1,\nabla X_2)}{|x|},\,
h_2 :=\frac{(-x_2,x_1)\cdot (\nabla X_1,\nabla X_2)}{|x|},\\
k:=\frac{(y_1,y_2)\cdot (\nabla Y_1,\nabla Y_2)}{|y|}\,.
\end{gather*}

It is an easy but important remark that because of orthogonality of $\nabla Y_1$ and $\nabla Y_2$ and 
the fact that $d\langle Y_1\rangle=d\langle Y_2\rangle$
we have
\begin{equation}
\label{kK1}
|k|^2=d\langle Y_1\rangle\,.
\end{equation}

Using the identities
\begin{align*}
\tilde{U}_{x_1x_2}&=U_{xx}\cdot \frac{x_1x_2}{|x|^2}-U_{x} \cdot\frac{x_1x_2}{|x|^3},
\\
\tilde{U}_{x_ix_i}&=U_{xx}\cdot \frac{x_i^2}{|x|^2}+U_{x} \cdot\frac{x_{3-i}^2}{|x|^3},
\qquad 1\le i\le 2,\\
\tilde{U}_{y_iy_i}&=U_{yy}\cdot \frac{y_i^2}{|y|^2}+U_{y} \cdot\frac{y_{3-i}^2}{|y|^3},
\qquad 1\le i\le 2,\\
\tilde{U}_{x_iy_j}&=U_{xy}\cdot \frac{x_iy_j}{|x|\cdot |y|},
\qquad 1\le i,j\le 2.
\end{align*}
and the property \eqref{kK1}, 
we can rewrite the condition \eqref{cond-U1} (for $x,y>0$) as
$$
U_{xx}|h_1|^2+\frac{U_x}{x}|h_2|^2 +2U_{xy}(h_1\cdot k) + (U_{yy}+\frac{U_y}{y})|k|^2 \le 0
$$
for all vectors $h_1$, $h_2$ and $k$ satisfying
\begin{equation}
\label{cone}
|h_1|^2+|h_2|^2 \le |k|^2,
\end{equation}
or, equivalently, as
\begin{equation}
\label{conditionU}
U_{xx}|h_1|^2+\frac{U_x}{x}|h_2|^2 +2|U_{xy}|\cdot |h_1|\cdot |k| + (U_{yy}+\frac{U_y}{y})|k|^2 \le 0.
\end{equation}

Denote $A=U_{xx}-\frac{U_x}{x}$, $B=U_{yy}+\frac{U_y}{y}$, and consider three cases:
\medskip 

\noindent Case (1): $A<0$ and $\beta_0:=|U_{xy}/A|\le 1$.
Let 
$$
\beta^2= \frac{|h_1|^2+|h_2|^2}{|k|^2}\,.
$$

Then we can rewrite our expression \eqref{conditionU} (for $|k|>0$, which is the only interesting case) as
$$
A |k|^2\bigg[\bigg(\frac{|h_1|}{|k|} - \beta_0\bigg)^2 -\frac{U_{xy}^2 -A(\beta^2\frac{U_x}{x}+B)}{A^2}\bigg].
$$
To maximize this expression we need to minimize the expression in the square brackets. If
$$
\beta\in [\beta_0, 1],
$$
we can always choose $k,h_1,h_2$ such that
$$
\beta^2= \frac{|h_1|^2+|h_2|^2}{|k|^2},\qquad \frac{|h_1|}{|k|}=\beta_0,
$$ 
which minimizes the expression in the square brackets. 

If
$$
\beta\in [0,\beta_0),
$$
we should make $|h_1|/|k|$ as close as possible to $\beta_0$  under the restriction 
$$
\beta^2= \frac{|h_1|^2+|h_2|^2}{|k|^2}.
$$ 
The best we can do is to put $h_2=0$. 

Conclusion: in case (1) the negativity of the expression in \eqref{conditionU} under the  condition \eqref{cone} is equivalent to
\begin{align*}
U_{xy}^2 -A(\beta^2\frac{U_x}{x} + B)) &\le 0,\qquad \beta\in 
[\,|U_{xy}/A|,1]\,,\\
U_{xx}\beta^2 + 2|U_{xy}|\beta+B &\le 0,\qquad \beta\in (0,
|U_{xy}/A|\,]\,.
\end{align*}
\medskip

\noindent Case (2):  $A<0$ and $\beta_0>1$.
We still need the minimum for the expression in the brackets 
above. This means that we should make $|h_1|/|k|$ as close 
as possible to $\beta_0$ under the restriction $\beta^2= \frac
{|h_1|^2+|h_2|^2}{|k|^2}$. The best we can do is to put 
$h_2=0$. 

Conclusion: in case (2) the negativity of the expression in 
\eqref{conditionU} under the condition \eqref{cone} is 
equivalent to
\begin{equation*}
U_{xx}\beta^2 + 2|U_{xy}|\beta+B\le 0,\qquad \beta\in [0,1]\,.
\end{equation*}
\medskip

\noindent Case (3): $A\ge 0$. Our expression becomes
$$
A\|k\|^2\bigg[ \bigg(\frac{|h_1|}{|k|} + \beta_0\bigg)^2 \!-\frac{U_{xy}^2 -A(\beta^2\frac{U_x}{x} + B) }{A^2}\bigg].
$$
Now we maximize the expression in the brackets under the restriction 
$\beta^2= \frac{|h_1|^2+|h_2|^2}{|k|^2}$. The best we can do is to put $h_2=0$. 

Conclusion: in case (3) the negativity of the expression in \eqref{conditionU} under the condition \eqref{cone} is equivalent to
\begin{equation*}
U_{xx}\beta^2 + 2|U_{xy}|\beta+B \le 0,\qquad \beta\in [0,1]\,.
\end{equation*}

Now we see that condition \eqref{conditionU} can be split into the following two. 

For every $0\le \beta\le 1$,
if $A<0$ and $|U_{xy}|\le \beta |A|$, then 
\begin{equation}
\label{Ineq1}
U_{xy}^2 -A\Bigl(\beta^2\frac{U_x}{x} + B\Bigr) \le 0,
\end{equation}
otherwise,
\begin{equation}
\label{Ineq2}
U_{xx}\beta^2 + 2|U_{xy}|\beta+B \le 0.
\end{equation}
These two conditions are equivalent when $A<0$ and $|U_{xy}|=\beta |A|$. Let us look at \eqref{Ineq1} first. If $U_x<0$, then the maximum value is attained for the smallest possible $\beta$ which is $|U_{xy}/A|$, hence \eqref{Ineq1} is contained in \eqref{Ineq2} in this case. When $U_x\geq 0$, the maximum value is attained for the largest possible $\beta$ which is $1$. Thus, \eqref{Ineq1} can be replaced by 
\begin{equation}\label{U-condition1}
\begin{cases}
|U_{xy}|\le \frac{U_x}{x}-U_{xx} \\
\textrm{and } U_x >0
\end{cases}
\Rightarrow
\frac{U_{xy}^2}{\frac{U_x}{x}-U_{xx}} +\frac{U_x}{x} + (U_{yy}+\frac{U_y}{y}) \le 0.
\end{equation}
 
The left hand side of inequality \eqref{Ineq2} is the quadratic function in $\beta$:
$$
h(\beta)=U_{xx}\beta^2 + 2|U_{xy}|\beta+
\bigl(U_{yy}+\frac{U_y}{y}\bigr).
$$
If $U_{xx}\geq 0$, then the maximum on $[0,1]$ occurs at $\beta=1$. Suppose $U_{xx}<0$. Then the maximum on $[0,\infty)$ is at $\beta=\frac{-|U_{xy}|}{U_{xx}}\ge 0$. If $\frac{-|U_{xy}|}{U_{xx}}> 1$, then again the maximum of $h$ on $[0,1]$ is 
attained at $\beta=1$. If $\frac{-|U_{xy}|}{U_{xx}}\le 1$, then the maximum is at $\beta=\frac{-|U_{xy}|}{U_{xx}}$. Hence, inequality \eqref{Ineq2} is equivalent to the following
conditions:
\begin{gather}
\label{U-cond1}
\begin{cases}
U_{xx}\ge 0\textrm{ or } \\
-|U_{xy}|< U_{xx}< 0
\end{cases}
 \Rightarrow \begin{cases} 
U_{xx} + 2U_{xy} +(U_{yy} + \frac{U_y}{y}) \le 0 \\
U_{xx} - 2U_{xy} +(U_{yy} + \frac{U_y}{y}) \le 0, 
\end{cases}\\
\label{U-cond2}
U_{xx}\le -|U_{xy}| \le 0 \Rightarrow U_{xy}^2 - U_{xx}(U_{yy} + \frac{U_y}{y}) \le 0. 
\end{gather}
Thus, the expression in \eqref{conditionU} is negative under the condition \eqref{cone} if and only if the implications 
\eqref{U-condition1}, \eqref{U-cond1}, and \eqref{U-cond2} hold.

\section{A simplified setting: $X$ real, $\beta=1$}\label{simplecase}

In the previous section, we worked with the case when both 
\newline\noindent$\nabla X_1, \nabla X_2\in \mathbb R^2$ and $\beta\in [0,1]$. 
Let us assume now that $\nabla X_2=0$, $|\nabla X_1|=|\nabla Y_1|$; we can restrict ourselves to the case $x_2=0$. Then $|h_1|=|k|$, $h_2=0$, and condition 
\eqref{conditionU} reduces to
\begin{align}
U_{xx} + 2U_{xy} +\Bigl(U_{yy} + \frac{U_y}{y}\Bigr) &\le 0 \label{U-simple1}, \\
U_{xx} - 2U_{xy} +\Bigl(U_{yy} + \frac{U_y}{y}\Bigr) &\le 0 \label{U-simple2}.
\end{align}
In many similar situations (see \cite{Bu3}, \cite{BaJ1} and \cite{BJV}), the best majorant in the simplified setting is also the best one in the general case. Hence, we may hope for the same effect in our problem and look first for functions $U$ satisfying \eqref{U-simple1} and \eqref{U-simple2}. 
We will proceed as follows.
\begin{enumerate}
\item Use the homogeneity of $U(x,y)$ to reduce the partial differential inequalities to ordinary differential inequalities for a one variable function $g(r)$.
\item Assume that the optimal $U$ (and $g$) will  solve (with equality) one of the two differential equations, wherever it is above the boundary $V$. Then solve the easier looking equation, which will be the one with $-U_{xy}$.
\item We will find the smallest constant $c$ for which there exists a majorant satisfying \eqref{U-simple2}. It will turn out that $U_{xy}\le 0$ for this solution, and hence \eqref{U-simple1} holds as well. 
\end{enumerate}

\subsection{Homogeneity and reduction in variables}
The function $U$ satisfies the same homogeneity condition as $V$: for all $t\in\bR$,
\[ U(tx,ty) = t^p U(x,y).\] To see this, suppose that $U$ is a suitable majorant of $V$. Then $U_t(x,y) =\frac{1}{t^p}U(tx,ty)\geq \frac{1}{t^p}V(tx,ty)=V(x,y)$ is also a majorant and as can be easily checked, satisfies \eqref{U-simple1} and \eqref{U-simple2}. Therefore $U_t$ is also a suitable majorant for each $t>0$. Now take the infimum over all $t$ to get a suitable majorant satisfying the homogeneity condition.

Define
\begin{equation*}
g(r) = U(1-r,r), \hspace{2cm} 0\le r\le 1.
\end{equation*}
Then 
\begin{equation}\label{defU}
U(x,y) = (x+y)^pU\Bigl(1-\frac{y}{x+y},\frac{y}{x+y}\Bigr) = (x+y)^pg\Bigl(\frac{y}{x+y}\Bigr).
\end{equation}
Set
\begin{gather*}
\mathcal{L}_pg(r):=rg''(r)+(1-r)g'(r)+pg(r),\\
\mathcal H_pg(r):= -r(1-r)g''(r)+(p-1)(1-2r)g'(r)+p(p-1)g(r).
\end{gather*}

Substituting \eqref{defU} into \eqref{U-simple1} and \eqref{U-simple2} gives the following equivalent conditions on $g$:
\begin{gather}\label{g-cond1}
\mathcal{L}_pg(r)+4r\mathcal H_pg(r)\le 0,\\
\label{g-cond2}
\mathcal{L}_pg(r)\le 0.
\end{gather}
The operator $\mathcal{L}_p$ is the Laguerre operator, the equation $\mathcal{L}_pf=0$ is the Laguerre equation and its solutions are the Laguerre functions. The function $g$ should 
also majorize the obstacle function $v_c$: 
\begin{equation*}
g(r) \geq v_c(r)=(1-r)^p-c^pr^p.
\end{equation*}
Finally note that for $0\le r\le 1$,
\begin{equation}\label{U_xy}
U_{xy}(1-r,r)=\mathcal H_pg(r).
\end{equation}
Since $v(0)=1$ for all $c$, we have $g(0)\ge 1$. As $g$ is the minimal possible function, it is likely that it solves either 
$\mathcal{L}_pg=0$ or $\mathcal{L}_pg(r)+4r\mathcal H_pg(r)=0$ wherever $g>v$. We consider first the simpler equation $\mathcal{L}_pg=0$ and attempt to construct $g$ using its solutions.

\subsection{The Laguerre functions}\label{laguerresection}
Just as for the Legendre case in \cite{BJV}, solutions to the Laguerre equation
\begin{equation}\label{Lag}
 xy'' +(1-x)y'+py =0
\end{equation}
are linear combinations of two independent solutions $L_p$ and $\tilde{L}_p$.
\begin{align}
L_p(x) = & 1 -px+\frac{p(p-1)}{4}x^2\notag \\
&+ \ldots +(-1)^n \frac{p(p-1)\cdots (p-n+1)}{n!^2}x^n +\ldots\,, \label{f1}\\
\tilde{L}_p(x) = & L_p(x)\log\frac{1}{|x|} + H(x), \label{f2}
\end{align}
$H$ is analytic in a neighborhood of $0$. Evidently, $L_p(x)$ is a bounded analytic function in $[0,1]$ and $\tilde{L}_p$ is unbounded near $0$. Denote by $z_p$ the smallest zero of $L_p$ on the interval $[0,1]$.
 
\begin{lemma}\label{fplemma}
For every solution to the Laguerre equation \eqref{Lag}, its smallest zero in $[0,1]$ is at most $z_p$.
\end{lemma}

\begin{proof}
Notice that $\tilde{L}_p(0) = +\infty$. Consider the Wronskian $W(x) = \tilde{L}_p'(x)L_p(x)-L_p'(x)\tilde{L}_p(x)$. 
By \eqref{f2}, we have
$$
W(x) = \frac{-L_p^2}{x}+H'L_p-L_p'H,
$$
which is strictly negative for $x$ close to $0$. Since $W'(x) =- \frac{1-x}{x}W(x)$, $W$ preserves sign in $[0,1]$ and is strictly negative. Since $L_p$ changes sign at $z_p$ 
from positive to negative, we have $L_p'(z_p)<0$ and, hence,  
$$
W(z_p) = -L_p'(z_p)\tilde{L}_p(z_p) = |L_p'(z_p)|\tilde{L}_p(z_p).
$$ 
Since $W<0$, it follows that $\tilde{L}_p(z_p)<0$. Now consider $f= c_1L_p+c_2\tilde{L}_p$ for $c_2>0$. Then $f(z_p)<0$ and $f(x)>0$ for $x$ close to $0$. Therefore $f$ has a zero in $(0, z_p)$. The same arguments work for $c_2<0$.
\end{proof}

\begin{lemma}\label{strictly-convex}
The function $L_p$ is strictly convex on $(0,z_p]$ for $1<p<\infty$; it is strictly concave on $(0,z_p]$ for $0<p<1$.
\end{lemma}

\begin{proof} First consider the case $1<p<\infty$. Starting with the Laguerre equation and then differentiating it, we get
\begin{gather}
\label{f1cond}
 xL_p'' + (1-x)L_p' +pL_p =0,\\
\label{f2cond}
xL_p''' +(2-x)L_p''+(p-1)L_p' = 0.
\end{gather}
Then $L_p''(0)=\frac{p(p-1)}{2}>0$.
Let $x_1>0$ be the first positive point where $L_p''(x_1)=0$. Suppose that $x_1<z_p$. Then $L_p(x_1)>0$ and 
\eqref{f1cond} implies that $L_p'(x_1)<0$. Then \eqref{f2cond} yields $L_p'''(x_1)>0$ and so $L_p''$ is strictly increasing at $x_1$ which is not possible. Therefore $x_1>z_p$ and $L_p$ is strictly convex on $(0,z_p]$. A similar argument shows that $L_p$ is strictly concave on $(0,z_p]$ for $0<p<1$.
\end{proof}

\begin{lemma}\label{sL-L}
$(sL_p')' =-pL_{p-1}$, $sL_p'=p(L_p-L_{p-1})$.
\end{lemma}

\begin{proof}
Differentiating the Laguerre equation 
\begin{equation}\label{Lag1}
(sL_p')'-sL_p'+pL_p = 0
\end{equation}
gives us 
$$
(sL_p')''-(sL_p')'+pL_p' = 0.
$$
Multiply this by $s$ and differentiate again to get
\[ s(sL_p')'''+(1-s)(sL_p')''+(p-1)(sL_p')' = 0.\]
Thus, $(sL_p')'$ solves the Laguerre equation with constant 
$p-1$ and hence is a multiple of $L_{p-1}$. It remains to use that $(sL_p')'(0)=-p$.

To get the second identity just apply the Laguerre equation.
\end{proof}

From now on in this subsection we assume that $p>1$.

\begin{lemma}\label{rootsorder} We have $L'_p<0$ on $(0,z_p]$, $z_{p}<z_{p-1}$.
\end{lemma}

\begin{proof}
Since $z_{p-1}$ is the root of $L_{p-1}$, by Lemma~\ref{sL-L}, we have 
\newline\noindent
$(sL_p')'(z_{p-1})=0$. Then by \eqref{Lag1}, we get $-z_{p-1}L_p'(z_{p-1})+pL_p(z_{p-1})=0$. Since $L_p\ge 0$, $L''_p>0$ on $(0,z_p]$, we have
$L_p'<0$ on $(0,z_p]$, and it follows that $z_{p-1}>z_p$.
\end{proof}

The following results improves the assertion of 
Lemma~\ref{fplemma}. 

\begin{lemma}\label{onlyLp}
Let $f\in C^1[0,1]$ be a supersolution of the Laguerre equation,
$$
sf''(s)+(1-s)f'(s)+pf(s)\le 0
$$
in the sense of distributions. Then $f(z_p)\le 0$.
\end{lemma}

\begin{proof} Set
$$
T(s)=sL_p'(s)f(s)-sf'(s)L_p(s).
$$
Then $T(0)=0$,
\begin{gather*}
T'(s)=L_p'(s)f(s)-f'(s)L_p(s)+sL_p''(s)f(s)-sf''(s)L_p(s)\\
\ge
L_p'(s)f(s)-f'(s)L_p(s)-(1-s)L_p'(s)f(s)-pL_p(s)f(s)
\\+(1-s)f'(s)L_p(s)+pf(s)L_p(s)= T(s).
\end{gather*}
Therefore,
$$
T(x)\ge 0, \qquad x\in [0,1],
$$
and hence
$$
0\le T(z_p)=z_pL_p'(z_p)f(z_p),
$$
and $f(z_p)\le 0$.
\end{proof}

\begin{lemma}\label{L-order}  We have
$L_p<L_{p-1}$, $L'_p<0$ on $(0,z_{p-1}]$, $L_p<0$ on $(z_p,z_{p-1}]$. Furthermore, $L_p$ has exactly one root in $[0,z_{p-1}]$.
\end{lemma}

\begin{proof}
By Lemma~\ref{sL-L}, for $s>0$ we have $L_p(s)=L_{p-1}(s)$ if and only if $L_p'(s)=0$. 

Suppose that $L_p(x)=L_{p-1}(x)$ for some $x\in (0,z_{p-1})$. First of all, $L_p'<0$ on $(0,z_p]$. Therefore $x>z_p$. Next $L_{p-1}>0$ in $(z_p,z_{p-1})$ and $L_p<0$ in some interval $(z_p, z_p+\epsilon)$. If $L_p$ is positive 
at a point in $(z_p,z_{p-1}]$, then for some $y\in(z_p,z_{p-1}]$ we have $L_p(y)<0$, $L'_p(y)=0$,
and then $L_{p-1}(y)<0$ which is impossible.
\end{proof}

\begin{corollary} The function $L_p$ is convex in $(0,z_{p-1})$.
\end{corollary}

\begin{proof}
Lemma~\ref{strictly-convex} gives this in $(0,z_p]$. Suppose that $L_p''(x)=0$ for some 
$x\in(z_p,z_{p-1})$. By the previous lemma, $L_p(x)<0$, and the Laguerre equation implies that $L_p'(x)>0$. Since $L_p'(0)<0$, there is a point $y\in (0,x)$ such that $L_p'(y)=0$. This contradicts the previous lemma.
\end{proof}

\begin{lemma}  $0<z_p<1$.
\end{lemma}

\begin{proof} By Lemma \ref{rootsorder}, it suffices to verify that $L_p(1)<0$, $1<p\le 2$.
By \eqref{f1}, we have
\begin{gather*}
L_p(1) = 1 -p+\frac{p(p-1)}{4}+ \ldots +(-1)^n \frac{p(p-1)\cdots (p-n+1)}{n!^2} +\ldots=\\
(p-1)\bigl(-1+\frac{p}{4}+ \frac{p(2-p)}{3!^2}
\ldots +  \frac{p(2-p)(3-p)\cdots (n-1-p)}{n!^2} +\ldots\bigr).
\end{gather*}
Since $p(2-p)\le 1$, $1<p\le 2$, we get 
$$
\frac{L_p(1)}{p-1}\le-1+\frac{1}{2}+ \frac{1}{3!^2}
\ldots +  \frac{(n-2)!}{n!^2} +\ldots<-1+e-2<0.
$$
\end{proof}

\begin{lemma}\label{est03}  $z_p\le 2/(p+1)$.
\end{lemma}

\begin{proof} By \eqref{Lag1}, we have 
$$
\int_0^{z_p}(sL_p'(s))'\,ds=\int_0^{z_p}sL_p'(s)\,ds-
p\int_0^{z_p}L_p(s)\,ds.
$$
By convexity of $L_p$ on $[0,z_p]$, we get 
\begin{gather*}
L_p'(s)\le L_p'(z_p),\qquad 0\le s\le z_p,\\
\int_0^{z_p}L_p(s)\,ds\ge -L_p'(z_p)\frac{z^2_p}{2},
\end{gather*}
and hence,
$$
z_pL_p'(z_p)\le \frac{z^2_p}{2}L_p'(z_p)+p\frac{z^2_p}{2}L_p'(z_p).
$$
It remains to use that $L_p'(z_p)<0$.
\end{proof}

\subsection{The function $\mathcal H_pL_p$}\label{hlp}
If  
$$
\mathcal{U}(x,y) = (x+y)^p L_p(\frac{y}{x+y}),
$$
then by \eqref{U_xy} we have
\begin{equation}
\label{H-U}
\mathcal H_pL_p(s) = \mathcal{U}_{xy}(1-s,s).
\end{equation}

\begin{lemma}\label{H-sL}
$\mathcal H_pL_p(s) = s(sL_p')'''$.
\end{lemma}

\begin{proof} We start with the identities
\begin{align}
(sL_p')' &= sL_p''+L_p' = sL_p'-pL_p, \nonumber \\
(sL_p')'' &= (sL_p'-pL_p)' = sL_p''-(p-1)L_p' = -[(p-s)L_p'+pL_p], \nonumber \\
(sL_p')''' &= -[(p-s)L_p''+(p-1)L_p']. \label{SL'}
\end{align}
Since 
\begin{gather*}
\mathcal H_pL_p(s) = s(s-1)L_p''+(p-1)(1-2s)L_p'+p(p-1)L_p,\\ 
(p-1)[sL_p''+(1-s)L_p'+pL_p]=0, 
\end{gather*}
we get
$$
\mathcal H_pL_p(s) = -s[(p-s)L_p''+(p-1)L_p'].
$$
It remains to use \eqref{SL'}.
\end{proof}

\begin{proposition}\label{Hneg}
If $2< p<\infty$, then $\mathcal H_pL_p< 0$ on $(0,z_{p-1}]$. 
If $1<p<2$, then $\mathcal H_pL_p>0$ on $(0,z_{p-1}]$.
\end{proposition}

\begin{proof} 
Let $p>2$. By Lemma~\ref{strictly-convex}, $L_{p-1}$ is strictly convex on 
$[0,z_{p-1}]$. Therefore, by Lemma~\ref{sL-L}, $(sL_p')'$ is 
strictly concave in $[0,z_{p-1}]$, and by Lemma~\ref{H-sL}, 
$\mathcal H_pL_p<0$ in $(0,z_{p-1}]$. Similarly, the case $1<p<2$  
follows from the fact that $L_{p-1}$ is strictly concave in $[0,z_{p-1}]$.
\end{proof}

Thus, for $2< p<\infty$, $L_p$ satisfies both \eqref{g-cond1} and \eqref{g-cond2} on 
$[0,z_p]$. Therefore, by \eqref{H-U}, $\mathcal U$ satisfies $\eqref{U-simple1}$ and 
$\eqref{U-simple2}$ on 
$\{(x,y):\frac{y}{x+y}\in[0,z_p]\}$.

\subsection{The obstacle function $v_{c_p}$}\label{est07} 

Set
\begin{gather*}
c_p=\frac{1-z_p}{z_p},\\
v_{c_p}(s)=(1-s)^p-\Bigl(\frac{1-z_p}{z_p}\Bigr)^ps^p.
\end{gather*}

Then $v_{c_p}(z_p)=0$,
\begin{gather*}
\mathcal H_pv_{c_p}=0,\\
\mathcal L_pv_{c_p}(s)=sp\Bigl[(p-1)(1-s)^{p-2}-
p\Bigl(\frac{1-z_p}{z_p}\Bigr)^ps^{p-2}\Bigl],
\end{gather*}
and
\begin{align*}
\mathcal L_pv_{c_p}(s)&>0,\qquad 0\le s<s_p,\\
\mathcal L_pv_{c_p}(s)&<0,\qquad s_p<s\le 1,
\end{align*}
where
$$
\Bigl(\frac{1-s_p}{s_p}\Bigr)^{p-2}=\frac {p}{p-1}
\Bigl(\frac{1-z_p}{z_p}\Bigr)^p.
$$

\begin{lemma}\label{est01}
If $p>2$, then $s_p<z_p$.
\end{lemma}

\begin{proof} It suffices to verify that
\begin{equation}\label{est02}
\frac {p-1}{p}<\Bigl(\frac{1-z_p}{z_p}\Bigr)^2.
\end{equation}

First, suppose that $2<p<3$. Estimate \eqref{est02} is equivalent to 
$$
z_p<\frac{p}{p+\sqrt{p^2-p}}.
$$
By \eqref{f1}, it suffices to check that
\begin{gather*}
0>L_p\Bigl(\frac{p}{p+\sqrt{p^2-p}}\Bigr)\\= 
1-\frac{p^2}{p+\sqrt{p^2-p}}+
\frac{p(p-1)}{4}\cdot\frac{p^2}{(p+\sqrt{p^2-p})^2}\\-
\sum_{n\ge 3}\frac{p(p-1)(p-2)(3-p)\cdots (n-1-p)}{n!^2}
\cdot\frac{p^n}{(p+\sqrt{p^2-p})^n}.
\end{gather*}
This follows from the estimate
$$
0>1-\frac{p^2}{p+\sqrt{p^2-p}}+
\frac{p^3(p-1)}{4(p+\sqrt{p^2-p})^2}
$$
or equivalently, for $2<p<3$,
\begin{gather*}
(p^2-p-\sqrt{p^2-p})(p+\sqrt{p^2-p})>\frac{p^3(p-1)}{4}, \\
\iff 4(p^2-2p)\sqrt{p^2-p}>p^4-5p^3+8p^2-4p, \\
\iff 4\sqrt{p^2-p}>(p-1)(p-2).
\end{gather*}

Second, if $p\ge 3$, then we use that by Lemma~\ref{est03}, 
$z_p\le 2/(p+1)$.
Therefore,
$$
\Bigl(\frac{1-z_p}{z_p}\Bigr)^2\ge \Bigl(\frac{p-1}{2}\Bigr)^2>
\frac {p-1}{p}.
$$
\end{proof}

In a similar way we have 

\begin{lemma}\label{est05}
If $1<p<2$, then 
$$
\frac {p-1}{p}>\Bigl(\frac{1-z_p}{z_p}\Bigr)^2.
$$
\end{lemma}

\subsection{The touching points}\label{touchpo}
For large $a$, we have 
$$
aL_p(r)>v_{c_p}(r),\quad r\in [0,z_p),\qquad
aL'_p(z_p)<v'_{c_p}(z_p).
$$ 
Now we lower the value of $a$ until either (i)
the graph of $aL_p$ on $[0,z_p)$
first touches the graph of $v_{c_p}$ or (ii) 
$aL_p(0)=v_{c_p}(0)$ or (iii) $aL'_p(z_p)=v'_{c_p}(z_p)$.
In fact, the case (ii) reduces to the case (i) because 
$L_p(0)=v_{c_p}(0)=1$, and $L'_p(0)=v'_{c_p}(0)=-p$.

Let us analyze the case (i). The touching point $s$ satisfies 
the equalities
\begin{equation}\label{touching_point}
\begin{cases}
(1-s)^p-c_p^ps^p = a L_p(s), \\ 
-p(1-s)^{p-1}-pc_p^ps^{p-1}= aL_p'(s),
\end{cases}
\end{equation}
or, equivalently,
$$
\begin{cases}
-pc_p^ps^{p-1} = apL_p(s) + a(1-s)L_p'(s), \\
-p(1-s)^{p-1}= -apL_p(s) + asL_p'(s).
\end{cases}
$$
Hence,
$$
\frac{c_p^ps^{p-1}}{(1-s)^{p-1}}=\frac{pL_p(s)+(1-s)L_p'(s)}{-pL_p(s)+sL_p'(s)}
$$
which implies that
$$
c_p^p=\frac{(1-s)^pL_p'(s)+p(1-s)^{p-1}L_p(s)}{s^pL_p'(s)-ps^{p-1}L_p(s)} =: F(s).
$$
Next we differentiate the function $F$ and, by 
Proposition~\ref{Hneg}, obtain that
$$
F'(s)=\frac{p(1-s)^{p-2}}{s^p}\frac{L_p(s)\mathcal H_pL_p(s)}
{(sL_p'(s)-pL_p(s))^2}<0,\qquad 0<s<z_p.
$$
Since
$$
F(z_p)=\Bigl(\frac{1-z_p}{z_p}\Bigr)^p=c_p^p,
$$
we obtain that the case (i) is impossible.

Thus we have
   
\begin{theorem}\label{v-L}
For $c_p=\frac{1-z_p}{z_p}$ and for some $a_p>1$, the function $v_{c_p}$ touches $a_pL_p$ at $z_p$ and $v_{c_p}<a_pL_p$ on $[0,z_p)$.  
\end{theorem}

Let us define 
\begin{equation}\label{majorant1}
g(s) = 
\begin{cases} 
a_pL_p(s), \qquad 0<s\le  z_p, \\ 
v_{c_p}(s), \qquad z_p<s\le  1.
\end{cases}
\end{equation}
Then $g\in C^1[0,1]$. By Lemma~\ref{est01}, $\mathcal L_pg\le 0$. Furthermore, 
by Proposition~\ref{Hneg} we have $\mathcal H_pg\le 0$.
Therefore, $g$ majorizes the obstacle function $v_{c_p}$ and 
satisfies \eqref{g-cond1} and \eqref{g-cond2}. 
Thus, the majorant $U(x,y)=(x+y)^pg(\frac{y}{x+y})$ satisfies 
\eqref{U-simple1} and \eqref{U-simple2}.

\subsection{Sharpness of the constant}
It remains to indicate that for any $c<\frac{1-z_p}{z_p}$,    
the function $v_c$ has no majorant satisfying \eqref{g-cond1} 
and \eqref{g-cond2}. Note that 
for $c<\frac{1-z_p}{z_p}$, $v_c(z_p)>0$. So any possible 
supersolution $f$ of the Laguerre equation, such that 
$f\ge v_c$ satisfies the inequality $f(z_p)>0$. However, this  
contradicts to Lemma~\ref{onlyLp}.
Since the Bellman function (which has the 
best constant)  satisfies the corresponding quadratic form 
inequalities, it follows that our constant is sharp.

\section{The general case, $2<p<\infty$}\label{generalcase}

Let $c_p=\frac{1-z_p}{z_p}$. In Section~\ref{simplecase}, we consider
conformal martingales $Y=(Y_1,Y_2)$ and real martingales $X$ satisfying 
$d\langle X\rangle=d\langle Y_i\rangle$, and established the 
sharp estimate
$$
\|X\|_p \le \Bigl(\frac{1-z_p}{z_p}\Bigr)\|Y\|_p,
$$
where $z_p$ is the smallest root of the bounded on $(0,1)$ Laguerre function $L_p$. We started with $V(x,y)=x^p-c_p^p y^p$ and found a majorant $U(x,y)$ satisfying the required quadratic-form inequalities \eqref{U-simple1} and \eqref{U-simple2}.
Now we turn to the general case where $X$ is a complex valued martingale and 
$d\langle X\rangle \le d\langle Y_i\rangle$. 
The function $U$ should satisfy \eqref{U-condition1} and \eqref{U-cond2}, in addition to \eqref{U-simple1} and \eqref{U-simple2} (or more precisely, in addition to \eqref{U-cond1}). We will show that the function $U$ obtained in the simple setting 
in Section~\ref{simplecase} works also in the general case. Henceforth, $U$ will denote this function, and $g$ will be its corresponding one-dimensional function defined in \eqref{majorant1}.

Recall conditions \eqref{U-condition1} and \eqref{U-cond2}:
\begin{gather}
U_x >0 \textrm{\ and\ }  |U_{xy}|\le \frac{U_x}{x}-U_{xx}
\hspace{6 cm} 
\notag\\
\Longrightarrow
U_{xy}^2 +\Bigl(\frac{U_x}{x} + \Bigl(U_{yy}+\frac{U_y}{y}\Bigr)\Bigr)\Bigl(\frac{U_x}{x}-U_{xx}\Bigr) \le 0\label{U-cond'},\\
\label{U-cond2'}
U_{xx}\le -|U_{xy}| \le 0 \Longrightarrow U_{xy}^2 - U_{xx}
\Bigl(U_{yy} + \frac{U_y}{y}\Bigr) \le 0.
\end{gather}
The following lemma shows that both these implications are trivially satisfied.

\begin{lemma}\label{est04}
For $x>0$, $U_x>0$ and $-U_{xx}< \frac{U_x}{x}-U_{xx}<|U_{xy}|$. 
\end{lemma}

This lemma implies that the `if' parts of \eqref{U-cond'} and 
\eqref{U-cond2'} do not hold for $x>0$. The special case when 
$x=0$ is also simple. Since $x$ corresponds to $1-r$, $x=0$ 
corresponds to $r=1$, where $g=v_{c_p}$ and hence $U=V$. Both 
$V_{xx}$ and $\frac{V_x}{x}$ are $0$ when $x=0$, and 
\eqref{U-cond'} and \eqref{U-cond2'} follow. 

\begin{proof}[Proof of Lemma~\ref{est04}]
When $U=V$, we have $V_x=px^{p-1}>0$ and 
$$
\frac{V_x}{x}-V_{xx}=-p(p-2)x^{p-2}<0=|V_{xy}|.
$$

From now on we assume that $U$ corresponds to the Laguerre function $g=a_pL_p$, 
$$
U(x,y)=(x+y)^pg\Bigl(\frac{y}{x+y}\Bigr),\qquad 0\le 
\frac{y}{x+y}<z_p.
$$
A simple computation shows that $U_x(1-s,s)=pg(s)-sg'(s)>0$ since $g$ and $-g'$ are strictly positive in $(0,z_p)$. 
It remains to show that
\begin{equation}\label{condition1}
\frac{U_x}{x}-U_{xx}< |U_{xy}|.
\end{equation}
By \eqref{H-U} and Propositon~\ref{Hneg}, $U_{xy}(1-s,s)=\mathcal H_pg(s)<0$ in $(0,z_p)$, 
and hence $|U_{xy}|=-U_{xy}$. Therefore, \eqref{condition1} is equivalent to 
\begin{equation}\label{condition2}
\frac{U_x}{x}-U_{xx} +U_{xy} <0.
\end{equation}

Furthermore, by the Laguerre equation, we have
\begin{align}
\Bigl(\frac{U_x}{x}\Bigr)(1-s,s) &= \frac{pg(s)}{1-s} - \frac{s}{1-s}g'(s), \label{Ux/x}\\
U_{xx}(1-s,s)&= s^2g''(s)-2(p-1)sg'(s)+p(p-1)g(s)= s^2g''(s) \nonumber \\ &
-2(p-1)sg'(s)-(p-1)sg''(s)-(p-1)(1-s)g'(s) \nonumber \\ &= -s(p-1-s)g''(s)-(p-1)(1+s)g'(s),\label{Uxx}\\
U_{xy}(1-s,s) &= \mathcal H_pg(s) = -s(p-s)g''(s)-s(p-1)g'(s). \label{Uxy}
\end{align}
By \eqref{Ux/x}, \eqref{Uxx} and \eqref{Uxy}, condition 
\eqref{condition2} is equivalent to 
\begin{multline*}
\frac{pg(s)-sg'(s)}{1-s}-sg''(s)+(p-1)g'(s)\\=\frac1{1-s}\bigl(\mathcal H_pg(s)
+(p-2)(sg'(s)-pg(s))\bigr)<0.
\end{multline*}
The latter inequality holds because $\mathcal H_pg$ and 
$sg'(s)-pg(s)$ are strictly negative on $(0,z_p)$.
\end{proof}

\section{Left hand side conformality, $1<p<2$}

In this section, we show that the same methods extend to the case of left hand side conformality when $1<p<2$. Again for the sake of simplicity, we work with the case
$$
d\langle Y_i\rangle= \frac{1}{2}d\langle Y\rangle \le 
d\langle X\rangle. 
$$
With this condition, the constant corresponding to \eqref{constp<2} is 
$$
\frac{z_p}{1-z_p}.
$$
Let us begin with the special case $d\langle X\rangle=
d\langle Y_i\rangle$. The obstacle functions are 
$$
V_c(x,y) = y^p-c^px^p,\hspace{5mm} v^*_c(s)=s^p-c^p(1-s)^p,
$$
and the majorants 
$$
U(x,y)=(x+y)^pg\Bigl(\frac{y}{x+y}\Bigr),\qquad g(s)=U(1-s,s)
$$ 
should satisfy the quadratic form inequalities \eqref{U-simple1}, \eqref{U-simple2},  \eqref{g-cond1}, and \eqref{g-cond2}. The function $v^*_c$ takes the value $-c^p$ at $s=0$ and increases to $1$ at $s=1$.  
We start with
$$
c_p=\frac{z_p}{1-z_p}
$$
and the function $v^*_{c_p}$. Arguing as in 
Subsection~\ref{est07} (see Lemma~\ref{est05}) we obtain that
$$
\mathcal L_pv^*_{c_p}(s)<0,\qquad z_p<s\le 1.
$$
For $a=0$ we have 
$$
aL_p(r)>v^*_{c_p}(r),\quad r\in [0,z_p),\qquad
aL'_p(z_p)<v^{*\prime}_{c_p}(z_p).
$$ 
Now we lower the value of $a$ until either (i)
the graph of $aL_p$ on $[0,z_p)$
first touches the graph of $v^*_{c_p}$ or (ii) 
$aL_p(0)=v_{c_p}(0)$ or (iii) $aL'_p(z_p)=v^{*\prime}_{c_p}(z_p)$.
In fact, the case (ii) reduces to the case (i) because 
$-c_p^pL_p(0)=v^*_{c_p}(0)$, and 
$-c_p^pL'_p(0)=v^{*\prime}_{c_p}(0)=pc_p^p$.

Let us analyze the case (i). The touching point $s$ satisfies 
the equalities
$$
\begin{cases}
s^p-c_p^p(1-s)^p = a L_p(s), \\ 
ps^{p-1}+pc_p^p(1-s)^{p-1}=aL_p'(s).
\end{cases}
$$
Put $\tilde{c}_p=\frac{1}{c_p}$ and 
$\tilde{a}=-\frac{a}{c_p^p}$, and obtain  
\begin{equation}\label{touching_point3}
\begin{cases}
(1-s)^p-\tilde{c_p}^ps^p = \tilde{a} L_p(s), \\ 
-p(1-s)^{p-1}-p\tilde{c_p}^ps^{p-1}= \tilde{a}L_p'(s).
\end{cases}
\end{equation}
Note that \eqref{touching_point3} is similar to 
\eqref{touching_point}. Arguing as in Subsection~\ref{touchpo}, we have
\begin{gather*} 
\tilde{c}_p^p = \frac{(1-s)^pL_p'(s)+p(1-s)^{p-1}L_p(s)}{s^pL_p'(s)-ps^{p-1}L_p(s)}=:F(s),\\
F'(s) = \frac{p(1-s)^{p-2}}{s^p}\frac{L_p(s)\mathcal H_pL_p(s)}{(sL_p'(s)-pL_p(s))^2}.
\end{gather*}

By Proposition~\ref{Hneg}, $\mathcal H_pL_p>0$,
and hence, $F'>0$ in $(0,z_p)$. 
Since
$$
F(z_p)=\Bigl(\frac{1-z_p}{z_p}\Bigr)^p=\tilde{c}_p^p,
$$
we obtain that the case (i) is impossible.

Thus, in the case $1<p<2$, 
for $c_p=\frac{z_p}{1-z_p}$ and for some $a_p<0$, the function $v^*_{c_p}$ touches $a_pL_p$ at $z_p$ and $v^*_{c_p}<a_pL_p$ on $[0,z_p)$.  

The best majorant satisfying the required quadratic form inequalities is
$$
g(s) = 
\begin{cases} 
a_pL_p(s), \qquad 0<s\le  z_p, \\ 
v^*_{c_p}(s), \qquad z_p<s\le  1.
\end{cases}
$$

\subsection{Sharpness of the constant}
Once again, for any $c<\frac{1-z_p}{z_p}$,    
the function $v_c$ has no majorant satisfying 
\eqref{g-cond2}. 
Indeed, $v_c(z_p)>0$, and any 
supersolution $f$ of the Laguerre equation such that 
$f\ge v_c$ satisfies the inequality $f(z_p)>0$. However, this contradicts to 
Lemma~\ref{onlyLp}.
Since the Bellman function (which has the 
best constant)  satisfies the corresponding quadratic form 
inequalities, it follows that our constant is sharp.

\subsection{The general case}
For the left hand side conformality with $1<p<2$, the general quadratic form requirement is
\begin{equation}\label{Ucondition1}
U_{xx}|h_1|^2+\frac{U_x}{x}|h_2|^2+2U_{xy}(h_1\cdot k)+(U_{yy}+\frac{U_y}{y})|k|^2 \le 0
\end{equation}
for all vectors $h_1$, $h_2$ and $k$ satisfying
$$
|k|^2 \le |h_1|^2+|h_2|^2.
$$
Setting $|k|=1$, $a=|h_1|$, and $\beta=(|h_1|^2+|h_2|^2)^{1/2}$, we obtain an equivalent form: 
\begin{multline}\label{cond-onab}
-\Bigl(\frac{U_x}{x}-U_{xx}\Bigr)a^2 + 2|U_{xy}|a + \frac{U_x}{x}\beta^2 +\Bigl(U_{yy}+\frac{U_y}{y}\Bigr) \le 0,\\
\text{for all\ }\beta\ge\max(1,a).
\end{multline}

\subsubsection{The case when $X$ is real-valued and $h_2=0$}
In this case, \eqref{Ucondition1} becomes 
\begin{equation}\label{U-condi1}
U_{xx}\beta^2+2|U_{xy}|\beta+ \Bigl(U_{yy}+\frac{U_y}{y}\Bigr)\le 0,
\qquad \beta\ge 1.
\end{equation}
When $U=V_c$, we have $U_{xx}= -c^p p(p-1)x^{p-2}<0$, $U_{xy}=0$, and therefore the maximum value in the left hand side is attained  for minimal $\beta =1$. Thus, \eqref{U-condi1} 
follows from \eqref{U-simple1}, \eqref{U-simple2} which hold
for $U$ as shown above.

Let us assume now that U corresponds to the Laguerre function
$g=a_pL_p$. Then 
$U(x,y)=(x+y)^pg(\frac{y}{x+y})$, and as above,
\begin{align*}
U_{xx}(1-s,s)&=s^2g''(s)-2(p-1)sg'(s)+p(p-1)g(s),\\
U_{xy}(1-s,s)&=-s(1-s)g''(s)+(p-1)(1-2s)g'(s)+p(p-1)g(s).
\end{align*}
Since $g\le 0$, $g'>0$, $g''<0$, and $\mathcal H_pg\le 0$ in $[0,z_p]$, we have 
$$
sg''(s)-(p-1)g'(s)< 0,\qquad 0\le s\le z_p,
$$
and hence,
$$
U_{xx}-U_{xy}\le 0,\qquad U_{xx}\le 0,\qquad U_{xy}\le 0.
$$
Therefore
$$
\frac{-|U_{xy}|}{U_{xx}}\le 1,
$$
and we obtain that the maximum value in the left hand side of \eqref{U-condi1} is attained at 
$\beta =1$ which is the special case considered in the previous section, so \eqref{U-condi1} is satisfied in this case.

Thus, $U$ always satisfies \eqref{U-condi1}. This completes the 
argument for the case when $X$ is real-valued, 
$d\langle Y_i\rangle\le d\langle X\rangle$, and $h_2=0$.

\subsubsection{The case when $X$ is complex-valued}
Now we deal with \eqref{cond-onab} in full generality.


If $\frac{U_x}{x}-U_{xx}\le 0$, then the maximal value 
of the expression in the 
left hand side of \eqref{cond-onab} for $a\le\beta$ and for 
fixed $\beta$ occurs when $a=\beta$, and we return to \eqref{U-condi1}. Therefore, from now on we assume that 
$\frac{U_x}{x}-U_{xx}>0$. This can happen only if $U>V=V_{c_p}$ since 
$\frac{V_x}{x}-V_{xx}=-c^p_p p(2-p)x^{p-2}<0$. Therefore, we can assume 
that $U>V$. Since $U_{xy}\le 0$, the maximal value of the expression in the 
left hand side of \eqref{cond-onab} 
as a function of $a\in[0,\infty)$
occurs at 
$$
a_*=\frac{|U_{xy}|}{\frac{U_x}{x}-U_{xx}}.
$$

As in the proof of Lemma~\ref{est04},
\begin{multline*}
\Bigl(\frac{U_x}{x}-U_{xx}+U_{xy}\Bigr)(1-s,s)
\\
=\frac1{1-s}\bigl(\mathcal H_pg(s)-(2-p)(sg'(s)-pg(s))\bigr)
\le 0
\end{multline*}
on $[0,z_p]$, and hence, $a_*\ge 1$.

If $\beta\le a_*$, then the maximal value in $a\in[0,\beta]$ 
in the 
left hand side of \eqref{cond-onab} is attained at $\beta$ and we return to \eqref{U-condi1}. If $a_*<\beta$, then the maximal value is at $a=a_*$, and it remains to verify that
\begin{equation}
U_{xy}^2+\Bigl(\frac{U_x}{x}\beta^2+U_{yy}+\frac{U_y}{y}\Bigr)\Bigl(\frac{U_x}{x}-U_{xx}\Bigr)\le 0.
\label{estw1}
\end{equation}

Since 
$$
U_x(1-s,s)=pg(s)-sg'(s)<0,
$$ 
the maximal value in the 
left hand side of \eqref{estw1} is attained for $\beta=a_*=a$ and we return once again to \eqref{U-condi1}
which completes the proof for the case of complex $X$.

\section{Estimates on $z_p$ and optimal constants in Theorem~\ref{MAINThm}}
\label{Asymptotics}

Let $J_0$ be the Bessel function of the zero order,
$$
J_0(x)=\sum_{n\ge 0}(-1)^n\frac{x^{2n}}{2^{2n}(n!)^2}.
$$
Denote its first positive zero by $j_0$. It is known (see, for example, 
\cite[Section 15.51]{WAT}) that
$$
j_0\approx 2.404826.
$$

Next, we use a Mehler-Heine type formula (see \cite[Theorem 8.1.3]{SZE}),
$$
\lim_{n\to\infty,\,n\in\mathbb N} L_n(x/n)=J_0(2\sqrt x).
$$
Arguing as in \cite[Section 10]{BJV} we conclude that
$$
\lim_{p\to\infty}pz_p=\frac{j^2_0}{4}.
$$
Hence,
$$
\lim_{p\to\infty}\frac{C_p}{p}=\frac{4\sqrt 2}{j^2_0}\approx
0.97815,
$$
where $C_p$ is the optimal constant in 
Theorem~\ref{MAINThm}~(2). 

Furthermore, by \eqref{f1}, for large $p$ we have
\begin{multline*}
0=L_{p'}(z_{p'})=1-(1+\varepsilon)(1-\delta)+
\frac{(1+\varepsilon)\varepsilon}{4}(1-\delta)^2 + \ldots\\ +
\frac{(1+\varepsilon)\varepsilon (1-\varepsilon)\ldots 
(n-2-\varepsilon)}{n!^2}(1-\delta)^n +\ldots\,,
\end{multline*}
where $\varepsilon=1/(p-1)$, $\delta=1-z_{p'}$,
and hence
$$
\delta-\varepsilon\Bigl(1-\sum_{n\ge 2}\frac{(n-2)!}{n!^2}\Bigr)=
O(\varepsilon^2+\delta^2),\qquad \delta\to 0+,\,
\varepsilon\to 0+.
$$

Denote 
\begin{equation}\label{sz01}
Q=1-\sum_{n\ge 2}\frac{(n-2)!}{n!^2}\approx 0.718282.
\end{equation}
Then
\begin{equation}\label{sz02}
\lim_{p\to\infty}p(1-z_{p'})=Q,
\end{equation}
and hence,
$$
\lim_{p\to\infty}\frac{C_{p'}}{p}=\frac1{Q\sqrt 2}\approx
0.98444,
$$
where $C_{p'}$ is the optimal constant in 
Theorem~\ref{MAINThm}~(1). 

Thus, we have 
$$
\lim_{p\to\infty}\frac{C_{p'}}{C_p}\approx
1.006,
$$

Finally, by \eqref{f1}, for $p>2$ we have
\begin{multline*}
L_{p'}\Bigl(1-\frac Qp\Bigr)=1-\frac p{p-1}\Bigl(1-\frac Qp\Bigr)+
\frac{p}{4(p-1)^2}\Bigl(1-\frac Qp\Bigr)^2 + \ldots\\ +
\frac{p(p-2)\ldots ((n-2)(p-1)-1)}{n!^2(p-1)^n}\Bigl(1-\frac Qp\Bigr)^n +\ldots\\<
-\frac {1-Q}{p-1}+
\frac{p}{4(p-1)^2}\Bigl(1-\frac Qp\Bigr)^2 + \ldots +
\frac{p(n-2)!}{n!^2(p-1)^2}\Bigl(1-\frac Qp\Bigr)^2 +\ldots\\=
\frac {1-Q}{p(p-1)^2}\bigl(-p^2+p+(p-Q)^2\bigr)<0.
\end{multline*}
Hence, $L_{p'}(1-\frac Qp)<0$, and thus 
\begin{equation}\label{sz03}
z_{p'}<1-\frac Qp.
\end{equation}

\section{Asymptotics for the $L^p$ norm of the Beurling--Ahlfors transform}

Here we formulate a new asymptotic result for the $L^p$ norm of the Beurling--Ahlfors transform $B$. Interestingly, Astala's theorem mentioned in Section~\ref{se2} would follow if it is only shown that $\|B\|_p$ has asymptotic order $p$; this was initially the motivation for looking for the norm. It remains an important sub-problem. Observe that (\ref{Beurest}) shows that 
$$
\limsup_{p\rightarrow\infty} \frac{\|B\|_p}{p} \le \sqrt{2}.
$$
This is presently the best published asymptotic information on $\|B\|_p$. In fact, the same one is obtained earlier in \cite{DV2}; it is proved there that 
\begin{equation}\label{DrVo1}
\|B\|_p  \le \sqrt{2}\,\tau_{p}(p^*-1),
\end{equation}
where
$$
\tau_p=\Bigl(\frac{1}{2\pi}\int_0^{2\pi} |\cos(\theta)|^pd\theta\Bigr)^{-1/p}.
$$
Although (\ref{DrVo1}) is worse overall than (\ref{Beurest}), the technique introduced in \cite{DV2} for obtaining this estimate leads us to the main result of this section. Let us state it next.
\begin{theorem}\label{main_thm}
Let $p>2$. Then
$$
\|B\|_p< \Bigl( \frac{p+3}{2}\pi \Bigr)^{1/(2p)}\cdot\frac{p-Q}{Q},
$$
where $Q$ is defined in \eqref{sz01}.
\end{theorem}
This estimate together with \eqref{sz02} gives 
$$
\limsup_{p\to\infty} \frac{\|B\|_p}{p} \approx 1.3922 < \sqrt{2}\,.
$$
Furthermore, since
$$
\Bigl( \frac{1003}{2}\pi \Bigr)^{1/2000}<1.004\text{\ \ and\ \ } 1.4\,Q>1.005,
$$
we have
$$
\|B\|_p<1.4\,p,\qquad p\ge 1000.
$$


To prove Theorem \ref{main_thm}, we use the martingale theory as found in 
\cite{BaWa1,NV1,BaMH} and a symmetry lemma from \cite{DV2}. We review this material in the following subsections.

\subsection{Martingales and the Beurling--Ahlfors transform}
Given a function $\varphi\in L^p(\bC)$, we denote the heat extension to $\bR^3_+$ by the same letter $\varphi$. If $Z(t)$ is the two dimensional Brownian motion starting at $z_0$, then the stochastic process $\varphi(Z(t), T-t) -\varphi(z_0, T)$ is a martingale denoted by
$$ 
(I\star\varphi)(t) = \int_0^t \nabla\varphi(Z(s), T-s)\cdot dZ(s).
$$
Given a $2\times 2$ matrix $A$, we denote the martingale transform of $I\star\varphi$ by the matrix $A$ as
$$
(A\star\varphi)(t) = \int_0^t A\,\nabla\varphi(Z(s), T-s)\cdot dZ(s).
$$
Note that we have hidden the implicit dependence on the starting point $z_0$. Each 
martingale transform $A\star\varphi$ can be projected back to $L^p(\bC)$ through the 
following procedure:
$$ 
T_A\varphi(z) = \lim_{T\to+\infty}\int_{\bC}
\mathbb E^{(z_0,T)}[(A\star\varphi)(T) | Z(T) = z]dm(z_0).
$$
That is, we first average over Brownian motion starting at $(z_0, T)$ and exiting $\bR^3_+$ at $(z, 0)$, then we average over all starting points $z_0$. This gives us the transformed function $T_A\varphi$. In fact, the projected operators correspond to second-order Riesz transforms as follows:
$$ 
A = \begin{pmatrix} a & b \\ c & d \end{pmatrix} \longleftrightarrow T_A = bR_2^2 - aR_1^2 + (c+d)R_1R_2.
$$
A quick calculation in \cite{Ja} shows that the projected singular integral operator has smaller norm than the martingale transform, i.e.
\begin{equation}\label{TAnorm}
\|T_A\varphi\|_p \le \|A\star\varphi\|_p.
\end{equation}
Furthermore, by a crucial theorem of Burkholder (see \cite{Bu2,BaMH}), we have 
\begin{equation}\label{Aphinorm}
\|A\star\varphi\|_p \le \|A\|(p^*-1)\|\varphi\|_p,
\end{equation}
where $\|A\|$ is the matrix norm of $A$. Combining (\ref{TAnorm}) with 
(\ref{Aphinorm}), we obtain the following estimate for the norm of the operator $T_A$:
$$
\|T_A\|_p \le \|A\|(p^*-1).
$$

Let 
$$
A_1 = \begin{pmatrix} 1 & 0 \\ 0 & -1 \end{pmatrix}, \quad
A_2 = \begin{pmatrix} 0 & 1 \\  1 & 0 \end{pmatrix},
$$
and 
$$
A^* = A_1 + i A_2 = \begin{pmatrix} 1 &  i \\  i & -1 \end{pmatrix}.
$$
These matrices correspond to the operators
$$
A_1\leftrightarrow T_1 = R_2^2-R_1^2,\quad 
A_2\leftrightarrow T_2 = -2R_1R_2, \quad A^* \leftrightarrow B.
$$ 
It is easy to see that 
\begin{align}
\|T_j\|_p &\le p^*-1, \label{norms} \\ 
\|B\|_p  &\le 2(p^*-1).\notag
\end{align}
It is proved in \cite{GMSS} that $\|T_j\|_p \geq p^*-1$ as well, hence equality holds
in \eqref{norms}. 

\subsection{The real part of $B\varphi$}
If $\varphi = \varphi_1 + i \varphi_2$, then the real part of $B\varphi$ is 
$$
\re(B\varphi) = T_1\varphi_1 - T_2\varphi_2.
$$
The above arguments applied here give 
\begin{equation}\label{root2}
\|\textrm{Re}(B\varphi)\|_p \leq \sqrt{2}(p^*-1)\|\varphi\|_p.
\end{equation}
Next we give the following important lemma of Dragicevic and Volberg \cite{DV2}; this is also crucial for our proof of the new asymptotic estimate on $\|B\|_p$. 

\begin{lemma}\label{DVlemma}
$\|B\|_p \leq \tau_p\sup_{\|\varphi\|_p=1}\|\textrm{Re}(B\varphi)\|_p$.
\end{lemma}

Observe that if we combine this estimate with \eqref{root2}, we obtain \eqref{DrVo1}.

\begin{proof}
Let $B_\theta$ denote the operator $e^{-i\theta}B$, $\theta\in [0, 2\pi)$. For any $z\in\bC$, observe that  
\begin{eqnarray*}
\textrm{Re}(B_\theta\varphi)(z) &=&  \textrm{Re}(B\varphi)(z)\cos\theta + \textrm{Im}(B\varphi)(z)\sin\theta \\
&=& |B\varphi(z)|\cos(\theta-\delta(z)),
\end{eqnarray*}
for some angle $\delta(z)$ depending on $z$. Taking the absolute value, raising to the $p$-th power and averaging over $\theta$ gives then 
$$
\frac{1}{2\pi}\int_0^{2\pi} |\textrm{Re}(B_\theta\varphi)(z)|^pd\theta = |B\varphi(z)|^p \tau_p^{-p}.
$$
Now integrate both sides with respect to $z$ to get
$$
\tau_p^{-p}\|B\varphi\|_p^p = \frac{1}{2\pi}\int_0^{2\pi} \|\textrm{Re}(B_\theta\varphi)\|_p^p d\theta.
$$
Since $B_\theta\varphi = B(e^{-i\theta}\varphi)$, it is clear that the norm-function 
$$
\sup_{\|\varphi\|_p=1}\|\textrm{Re}(B_\theta\varphi)\|_p = \sup_{\|\varphi\|_p=1} \|\textrm{Re}(B\varphi)\cos\theta + \textrm{Im}(B\varphi)\sin\theta\|_p
$$
is constant in $\theta$. Thus we have
\begin{eqnarray*}
\|B\|_p^p &=& \sup_{\|\varphi\|_p=1}\|B\varphi\|_p^p = \frac{\tau_p^p}{2\pi}\sup_{\|\varphi\|_p=1}\int_0^{2\pi}\|\textrm{Re}(B_\theta\varphi)\|_p^p d\theta \\
&\leq& \frac{\tau_p^p}{2\pi}\int_0^{2\pi}\sup_{\|\varphi\|_p=1}\|\textrm{Re}(B_\theta\varphi)\|_p^p d\theta \\
&=& \tau_p^p\sup_{\|\varphi\|_p=1}\|\textrm{Re}(B\varphi)\|_p^p.
\end{eqnarray*}
\end{proof}

\subsection{The use of conformality}
It is stated in \cite{BaWa1} and shown in \cite{BaJ1} that the martingale
$$
A^*\star\varphi 
= (A_1\star\varphi_1 - A_2\star\varphi_2) + i (A_2\star\varphi_1 + A_1\star\varphi_2)
$$
is a conformal martingale. 
Using this property, Ba\~nuelos and Janakiraman  \cite{BaJ1} establish that for $2\le p<\infty$,
$$
\left\|\frac{A^*\star\varphi}{2}\right\|_p \le \sqrt{\frac{p^2-p}{2}}\|\varphi\|_p,
$$
which then leads them to \eqref{Beurest} and \eqref{best-est}. 



\subsection{Proof of Theorem~\ref{main_thm}
}

Here we use the norm $\|B\|_{L^{p'}(\bC,\bR)}$ of the operator $B$ 
on the space of $L^{p'}$ integrable functions defined on $\mathbb C$, 
restricted to real valued functions.

\begin{lemma}\label{asym_lemma}
Let $1<p<\infty$. For $\varphi\in L^p(\bC)$, we have
$$
\|\re(B\varphi)\|_p \le \|B\|_{L^{p'}(\bC,\bR)}\|\varphi\|_p.
$$
\end{lemma}

\begin{proof}
In the following, $\psi$ always denotes real valued functions.
\begin{gather*}
\|\re(B\varphi)\|_p = \sup_{\|\psi\|_{p'}=1}\int (T_1\varphi_1 - T_2\varphi_2)\psi 
\\=\sup_{\|\psi\|_{p'}=1}\int (\varphi_1 T_1\psi -\varphi_2 T_2\psi)  
= \sup_{\|\psi\|_{p'}=1}\int (\varphi_1,\varphi_2)\cdot (T_1\psi, -T_2\psi) 
\\ \le \|\varphi\|_p \sup_{\|\psi\|_{p'}=1}\|B\psi\|_{p'}  
= \|\varphi\|_p\|B\|_{L^{p'}(\bC,\bR)}. 
\end{gather*}
\end{proof}

If $\psi$ is real valued, then $A^*\star\psi$ is a conformal martingale with quadratic variation $d\langle A^*\star\psi\rangle = 2\,d\langle I\star\psi\rangle$. 
Hence by \eqref{TAnorm} and by Theorem~\ref{MAINThm} we have 
$$
\|B\|_{L^{p'}(\bC, \bR)} \le \sup_{\|\psi\|_{p'}=1}\|A^*\star \psi\|_{p'} \le \frac{z_{p'}}{1-z_{p'}}.
$$
By Lemmas \ref{DVlemma} and \ref{asym_lemma}, we get 
\begin{equation}\label{zs04}
\|B\|_p \le \tau_p \|B\|_{L^{p'}(\bC,\bR)}\le\tau_p\frac{z_{p'}}{1-z_{p'}}.
\end{equation}
It remains to use Wallis' formula:
$$
\tau^{-2n}_{2n}=\frac{1}{2\pi}\int_0^{2\pi} |\cos(\theta)|^{2n}d\theta=
\Bigl(\frac{1\cdot3\cdot\ldots\cdot (2n-1)}{2\cdot4\cdot\ldots\cdot (2n)}\Bigr)^2>
\frac{2}{(2n+1)\pi}.
$$
By monotonicity of $\frac{1}{2\pi}\int_0^{2\pi} |\cos(\theta)|^{p}d\theta$ we obtain 
$$
\frac{1}{2\pi}\int_0^{2\pi} |\cos(\theta)|^{p}d\theta>\frac{2}{(p+3)\pi},
$$
and hence
$$
\tau_p<\Bigl(\frac{p+3}{2}\pi\Bigr)^{1/(2p)};
$$
this together with \eqref{zs04} and \eqref{sz03} proves Theorem~\ref{main_thm}.

\markboth{}{\sc \hfill \underline{References}\qquad}

\end{document}